\documentclass[12pt]{article}
\usepackage{graphicx}
\usepackage{cmap}					
\usepackage[T2A]{fontenc}			
\usepackage[utf8]{inputenc}		
\usepackage{hhline}
\usepackage{amsthm}
\usepackage{amsmath}
\usepackage{amssymb}
\usepackage{amsfonts}
\usepackage{verbatim}
\usepackage{multirow}
\usepackage{hhline}%
\usepackage{listings}

\usepackage{epstopdf}
\usepackage{subfigure}
\usepackage{multirow}
\usepackage{array}

\usepackage{blindtext}

\newcolumntype{C}[1]{>{\centering\let\newline\\\arraybackslash\hspace{0pt}}m{#1}}

\theoremstyle{plain}
\newtheorem{theorem}{Theorem}[section]
\newtheorem{corollary}[theorem]{Corollary}
\newtheorem{lemma}[theorem]{Lemma}

\newtheorem{alg}{Algorithm}

\theoremstyle{definition}

\theoremstyle{remark}

\newcommand{\Real}[0]{\operatorname{Re}}
\newcommand{\Imag}[0]{\operatorname{Im}}
\newcommand{\diag }[0]{\operatorname{diag }}

\newcommand{\card}[0]{\operatorname{card}}

\begin{document}



\title{Optimization of Power Network Frequency Control via Analytic Bounds on the Objective Function}

\author{
O. O. Khamisov\\
Center for Energy Systems,\\ Skolkovo Institute of Science and Technology,\\ Moscow, Russia
}
\date{}
\maketitle

 

\begin{abstract}
Frequency control in power networks is designed to maintain power balance by adjusting generation that allows to keep frequency at its nominal value (i.e. 50 Hz). If power disturbance occurs, it leads to frequency oscillations that are consequently suppressed by the control. Contemporary control schemes depend on a number of parameters. They can be adjusted to increase control efficiency. The main factor that defines control efficiency is maximal frequency deviation (nadir). However corresponding objective function (value of the nadir) is non-smooth and has multiple local extremums, thus it is difficult to optimize. The aim of this work is to present an analytic conservative estimate of the nadir and then optimize the estimate instead of optimizing nadir, since in practice global minimums of both functions coincide. Nadir and its estimate are non-smooth, therefore standard zero order method is applied. Numerical experiments show that the estimate has unique minimum while nadir has a set of local minimums. Therefore optimization of the estimate, unlike optimization of the nadir, not only computationally easy, but also returns desired global optimum. Such optimization allows to reduce nadir up to two times compared to the nadir with default control parameters. Moreover it allows to limit frequency oscillations in order to reduce wear of the equipment.
\end{abstract}

\section{Introduction}

Power networks are susceptible to power imbalances due to changes in power demand. Additionally generator failure may result in significant disturbance in power balance. As a result, electrical power drained from the generators exceeds mechanical power supply and generators start to slow down. Consequently frequency of the power network drops \cite{BV}, \cite{MBB}, \cite{WW}. Large frequency oscillations may result in equipment damage or its emergency shutdown. In order to counter this effects frequency control is used. It adjusts power generation in order to restore power balance and deliver frequency to its nominal value (50 Hz or 60 Hz).

It is known \cite{MBB}, that maximal frequency oscillation (nadir) happens during first 30 seconds after power disturbance appearance. Currently used frequency control utilises proportional controller called Droop control (or Primary Frequency control) to reduce this initial frequency drop. Droop control works at timescale of tens of seconds. During this period the system is most vulnerable due to frequency oscillations. There exist various other versions of frequency control schemes \cite{ZML}-\cite{ZMLB} however they are not implemented in power systems, therefore they are not considered within this paper.

Droop control has a set of control parameters, which are currently chosen to ensure system's stability. However it is possible to adjust these parameters to reduce frequency oscillations without loosing stability of the system. There is no agreed way to describe system's behavior with a set of particular parameters without simulations. In particular maximal frequency deviations (nadir) from the nominal value are the key factors, that influence system's reliability. Aim of the paper is development of approach, that would minimize maximal absolute value of frequency oscillations (nadir) without loss of system's stability.

Classical linearized model of transmission power network is considered. It is assumed that several buses suffer from a step change of power generation or consumption, which results in frequency oscillations. As a result dynamics of the network as defined by a system of linear differential equations
\begin{equation}\label{wdefwef}
  \dot x=A(r)x+P^D.
\end{equation}
Here vector $x(t)$ includes subvector of bus frequency deviations $\omega(t)$. It is required to find vector of droop control parameters $r=(r_1,\dots,r_n)^T$, that would minimize nadir $F(r)=\max_{l=\{1,n\}}\max_{t\ge0}|\omega_l(t)|$ for all buses $l\in N$, where $N$ is the set of buses. Primary frequency control, considered in this paper stabilizes frequency at equilibrium during the first several tens of seconds \cite{KN}. Therefore in numerical experiments maximization is done for the first $100$ seconds. Vector-function $\omega(t)$ is a part of the solution of linear differential equations, therefore it is given by oscillatory functions with infinite number of local extremum. Calculation of the nadir is done by d.c. approach which is computationally expensive. Nadir as a function of control parameters is non-smooth, therefore its optimization is done by Hooke and Jeevse zero order method. Finally nadir function $F(r)$ has multiple local extremum, therefore it is not possible to guarantee that global optimum will be found by such approach.

Alternatively, analytical estimate for nadir is derived in order to improve optimization efficiency and exclude local extremums. Firstly we derive frequency majorants $\mathcal{M}_l(t)\ge\omega_l(t)$ for each bus $l$ that in practice have unique optimum, therefore their maximum $G(r)=\max_{l=\{1,n\}}\max_{t\ge0}\mathcal{M}_l(t)\ge F(r)$ (estimation of the nadir) can be found via golden section method, which is computationally much faster, then d.c. optimization applied to calculate $F(r)$. Function $G(r)$ is non-smooth and is optimized by Hooke and Jeeves method. However in practice it has unique minimum which coincides with global minimum of $F(r)$. As a result optimization of $G(r)$ is computationally more efficient than optimization of $F(r)$, moreover optimization of $G(r)$ delivers global maximum for both $G(r)$ and $F(r)$, since $G(r)$ does not have local extremums.

Numerical experiments are done in order to compare efficiency of estimate minimization with direct optimization of nadir.

The article is organized in the following way. In section \ref{pr_st} network model and optimization aim are described.
In section \ref{maj_freq} estimate of the nadir is derived.
In section \ref{mn} optimization of both the nadir and its estimate are described.
In section \ref{num_res} results of numerical experiments are presented.
In section \ref{Conc} final observations and directions of the future work are discussed.
\section{Problem Statement}\label{pr_st}
\subsection{Notations}
Let $i$ be imaginary unit.
Let $\mathbb{R}$ be set of real numbers, cardinality of a finite set $S$ is defined as $\card(S)$.
For an arbitrary matrix (vector) $X$ its transpose is denoted by $X^T$.
For an arbitrary vector $x\in\mathbb{R}^n$ and set $I\subseteq\{1,\dots,n\}$ we define subvector $x_I$.
For an arbitrary matrix $X\in\mathbb{R}^{n\times}$ and set $I\subseteq\{1,\dots,n\}$ we define row submatrix $X_I$.
Let $\textbf{0}$ be zero matrix of the corresponding size, $I$ is identity matrix of the corresponding size, $0$ is zero vector of the corresponding size, $\rho$ is vector of ones of the corresponding size.
For vector $x=(x_1,\dots,x_n)$ we denote by $\diag(x)=\diag(x_1,\dots,x_n)$ diagonal matrix with elements $x_l,\;l\in\{1,\dots,n\}$.
Operations $\Real\,X$, $\Imag\,X$, $|X|$ are considered to be elementwise, if $X$ is a vector or a matrix.

\subsection{Model description}
Classical generator model \cite{BH}, \cite{WWW} is used. The power transmission network \cite{BV} is described by a directed graph $(N,E)$,  where $N$ is the set of $n$ buses defined by their indices: $N=\{1,\dots,n\}$, $E = \{(l,j),\;l,j\in N$ is set of $m$ lines. It is assumed, that the network is connected.

Dynamics of the power transmission network is defined by the system of linear differential equations \cite{ZML}-\cite{ZMLB}. Kron reduction \cite{K}-\cite{DB} is applied in order to exclude load buses, therefore system has the following form:
\begin{subequations}\label{msys}
\begin{align}
m_l\dot\omega_l=&-d_l\omega_l+\sum_{j:(j,l)\in E}p_{ji}-\sum_{j:(l,j)\in E}p_{lj}-\nonumber\\
&-p^M_l+p^D_l,\;\omega_l(0)=0,\;l\in N,\label{msys1}\\
\dot p_{lj}=&b_{lj}(\omega_l-\omega_j),\;p_{lj}(0)=0,\;(l,j)\in E,\label{msys2}\\
t^G_l\dot p^M_l=&-p_l^M+v_l,\;p^M_l(0)=0,\;l\in N,\label{msys3}\\
t^B_l\dot\psi_l=&-\psi_l-r_l\omega_l,\;\psi_l(0)=0,\;l\in N.\label{msys4}
\end{align}
\end{subequations}
Variables of the system have the following meanings:
\begin{itemize}
  \item $\omega_l,\;l\in N$ are deviations of bus frequencies from nominal value,
  \item $p_{lj},\;(l, j)\in E$ are line active power flows,
  \item $p^M_l,\;l\in N$ are mechanic power injections at generators,
  \item $\psi_l,\;l\in N$ are positions of valves,
\end{itemize}
Parameters of the system:
\begin{itemize}
    \item $m_l>0,\;l\in N$ are generators inertia constants,
    \item $d_l>0,\;l\in N$ are steam and mechanical damping of generators and frequency-dependent loads,
    \item $p^d_l,\;l\in N$ are unknown disturbances (assumed constant),
    \item $b_{lj}>0,\;(l, j)\in E$ are line parameters that depend on line susceptances, voltage magnitudes and reference phase angles,
    \item $t^G_l>0,\;l\in N$ are time constants that characterize time delay in fluid dynamics in the turbine,
    \item $t^B_l>0,\;l\in N$ are time constants that characterize time delay in governor response.
    \item $r_l>0,\;l\in N$ are control parameters.
\end{itemize}

Equations of the system:
\begin{itemize}
    \item Equations (\ref{msys1}) are generator swing equations,
    \item Equations (\ref{msys2}) are equations of direct current linearized power flows,
    \item Equations (\ref{msys3}) are turbine dynamics,
    \item Equations (\ref{msys4}) are governor dynamics.
\end{itemize}
Control parameters $r$ are chosen non-negative in order to ensure system's stability. It is possible to adjust this parameters in order to reduce maximal frequency deviations even further and still keep system stable.


\subsection{Matrix Representation}\label{mat_rep}
In order to introduce nadir estimate the system (\ref{msys}) is presented in the following form:
\begin{equation}\label{meq}\dot x=A(r)x+P^D,\;x(0)=0,\end{equation}
$$A(r)=\left(\begin{array}{cccc}
    -MD & -MC & M & \textbf{0} \\
    BC^T   & \textbf{0} & \textbf{0} & \textbf{0} \\
    \textbf{0} & \textbf{0} & -T^B & T^B \\
    -T^VR & \textbf{0} & \textbf{0} & -T^V \\
\end{array}\right),\;x=\left(\begin{array}{c}\omega(t)\\p(t)\\p^M(t)\\v(t)\end{array}\right),\;P^D=\left(\begin{array}{c}Mp^D\\0\\0\\0\end{array}\right).$$

Here $\omega(t)=(\omega_1(t),\dots,\omega_n(t))^T$, $p(t)=(p_1(t),\dots,p_m(t))^T$, $p^M(t)=(p^M_1(t),\dots,p^M_n(t))^T$, $v(t)=(v_1(t),\dots,v_n(t)$, $D=\diag (d_1,\dots,d_n)\succ0$, $M=\diag \left(1/m_1,\dots,1/m_n\right)\succ0$, $B=\diag (1/b_1,\dots,1/b_m)\succ0$, $T^G=\diag (1/t^G_1,\dots,1/t^G_n)\succ0$, $T^V=\diag (1/t^V_1,\dots,1/t^V_n)\succ0$, $R=\diag (r_1,\dots,r_n)\succ0$, $C\in\mathbb{R}^{n\times m}$ is the incidence matrix \cite{BRB} of the system graph $G$, $p^D=(p^D_1,\dots,p^D_n)^T$ is the inhomogeneity vector of disturbances.

\subsection{Optimization aim}
In the models, describing energy systems, matrix $A$ is diagonalizable \cite{IEEE}-\cite{PST}. We are using this observation in the further results.
Let 
\begin{equation}
    A(r) = S(r)\Lambda(r)S^{-1}(r)
\end{equation}
be its eigenvalues decomposition, where $\Lambda(r)$ diagonal matrix of eigenvalues that has the following form
$$\overline\Lambda(r) =\left(
                     \begin{array}{cc}
                        \Lambda^1(r) & \textbf{0} \\
                       \textbf{0} &  \textbf{0} \\
                     \end{array}
                   \right),$$
where $\Lambda^1$ is diagonal matrix of nonzero eigenvalues.
Let us define nadir as a function of control parameters:
\begin{equation}\label{ndf}
  F(r)=\max_{l\in N}\max_{t\in[0,t_1]}|\omega_l(t)|.
\end{equation}
Nadir minimization problem  has the following form:
\begin{subequations}
\begin{equation}\label{obj}
    \min_{r\ge0}F(r),
\end{equation}
s.t.
\begin{equation}\label{stab}
    \max\Re_{i\in N}\lambda_i(r)\le0,
\end{equation}
\begin{equation}\label{oc_con}
  \max_{l\in N}\left|\frac{\Real\,\lambda_l(r)}{\Imag\,\lambda_l(r)}\right|\ge\xi,\;\xi>0.
\end{equation}
\end{subequations}
where $\omega(t)$ is part of the solution of the system (\ref{msys}), that depends on $r$. Primary frequency control, considered in this paper stabilizes frequency at equilibrium during the first several tens of seconds \cite{KN}. Therefore in the experiments $t_1$ is taken equal to $100$ seconds. Constraint (\ref{stab}) is used to ensure system's stability, constraint (\ref{oc_con}) is used to suppress small frequency oscillations. Its effects are discussed in the numerical results section.

Functions $\omega_l(t)$ are oscillatory and have infinite amount of extremums as a solution of system of linear differential equations. In addition function $\max_{t\in[0,t_1]}|\omega_l(t)|$ is non-smooth, therefore calculation of any value of $F(r)$ is computationally difficult. In addition, as it is shown in numerical results section, function $F(r)$ has local minimums. Thus approximating nadir with  function $G(r)\ge F(r)$ of simpler structure is done. Then problem
\begin{equation}\label{objr}
\min_{r\ge0}G(r),
\end{equation}
is considered instead of (\ref{obj}).

\section{Estimate of the nadir}\label{maj_freq}
Firstly let us show that there always there always exists set of control parameters $R$ for which frequency deviations $\omega(t)$ converge to a constant value.
\begin{theorem}
Let $R=\textbf{0}$. System (\ref{msys}) is asymptotically stable over $\omega$.
\end{theorem}
\begin{proof}
If $R=\textbf{0}$ then matrix of the system (\ref{msys}) as an upper triangular matrix:
\begin{equation}
A(r)=\left(\begin{array}{cccc}\cline{1-2}
    \multicolumn{1}{|c}{-MD} & \multicolumn{1}{c|}{-MC}& M & \textbf{0}\\
    \multicolumn{1}{|c}{BC^T} & \multicolumn{1}{c|}{\textbf{0}} & \textbf{0} & \textbf{0}\\\cline{1-3}
    \textbf{0} & \textbf{0} & \multicolumn{1}{|c|}{-T^G} & T^G\\\cline{3-4}
    \textbf{0} & \textbf{0} & \textbf{0} & \multicolumn{1}{|c|}{-T^V}\\\cline{4-4}
\end{array}\right).
\end{equation}
Variables $\psi$ do not depend on any other variables, and is defined by the equation
\begin{equation}
\dot\psi=-\psi,    
\end{equation}
thus is asymptotically stable. Similarly $p^m$ is asymptotically stable as it depends only on $\psi$.
Diagonal block for the remaining variables $\omega$ and $p$ has form
\begin{equation}
\left(\begin{array}{cc}
    -MD & MC\\
    -BC^T & \textbf{0}\\
\end{array}\right)
\end{equation}
and the corresponding homogeneous system is
\begin{subequations}
\begin{align}
    \dot\omega =& -MD\omega+MCp,\\
    \dot p =& BC^T.
\end{align}
\end{subequations}
Let us consider the following Lyapunov function:
\begin{equation}
V(\omega,p) = 
\left(\begin{array}{c}
    \omega\\
    p\\
\end{array}\right)^T
\left(\begin{array}{cc}
    M^{-1} & \textbf{0}\\
    \textbf{0} & B^{-1}\\
\end{array}\right)\left(\begin{array}{c}
    \omega\\
    p\\
\end{array}\right).
\end{equation}
Its derivative is given by
\begin{equation}
    \dot V(\omega,p) = -2\omega^TD\omega.
\end{equation}
That according Barbashin-Krasovskii-LaSalle theorem \cite{HC} proves asymptotic stability of $\omega$.
\end{proof}
Let us introduce the new eigenvalues matrix $$\overline\Lambda(r) =\left(
                     \begin{array}{cc}
                        \Lambda^1(r) & \textbf{0} \\
                       \textbf{0} &  -I \\
                     \end{array}
                   \right),$$
and new system matrix
$$\overline A(r) = S(r)\overline\Lambda(r) S^{-1}(r).$$
Then the following lemmas can be derived:
\begin{lemma}\label{l35}
Let
$$x(t)=\left(\begin{array}{c}\omega(t)\\z(t)\end{array}\right)$$
be solution of the system (\ref{meq}) and
$$\overline x(t)=\left(\begin{array}{c}\overline\omega(t)\\\overline z(t)\end{array}\right)$$
be solution of the system
\begin{equation}\label{replace}
    \dot {\overline x}=\overline A(r)x+P^D,\;\overline x(0)=0.
\end{equation}
then
$$\omega(t)=\overline\omega(t)$$
\end{lemma}
\begin{proof}
Solution of (\ref{meq}) is given by
$$x(t) = \int_0^te^{A(r)(t-\tau)}P^Dd\tau = S(r)e^{\Lambda(r) t}\int_0^te^{-J\tau}d\tau S^{-1}(r)P^D =$$
$$= S(r)e^{\Lambda t}\left(
           \begin{array}{cc}
              (\Lambda^1(r))^{-1}(I-e^{-\Lambda^1(r)t}) & \textbf{0} \\
             \textbf{0} & It \\
           \end{array}
         \right)S^{-1}(r)P^D = S(r)\left(
           \begin{array}{cc}
              (\Lambda^1(r))^{-1}(e^{\Lambda^1(r)t}-I) & \textbf{0} \\
             \textbf{0} & It \\
           \end{array}
         \right)S^{-1}(r)P^D.$$
Similarly
$$\overline x(t) = S(r)\left(
           \begin{array}{cc}
              (\Lambda^1(r))^{-1}(e^{\Lambda^1(r)t}-I) & \textbf{0} \\
             \textbf{0} & I-e^{-t} \\
           \end{array}
         \right)S^{-1}(r)P^D.$$
Let us consider  their difference
$$x(t)-\overline x(t) = S(r)\left(
           \begin{array}{cc}
             \textbf{0} & \textbf{0} \\
             \textbf{0} & It-I+e^{-t} \\
           \end{array}
         \right)S^{-1}(r)P^D =$$
$$=\left(
           \begin{array}{cc}
             S^1_N(r) & \textbf{0} \\
             S_Z^1(r) & S_Z^0(r) \\
           \end{array}
         \right)\left(
           \begin{array}{cc}
             \textbf{0} & \textbf{0} \\
             \textbf{0} & It-I+e^{-t} \\
           \end{array}
         \right)S^{-1}(r)P^D =$$
$$=\left(
           \begin{array}{cc}
             \textbf{0} & \textbf{0} \\
             \textbf{0} & S_Z^0(r)(It-I+e^{-t}) \\
           \end{array}
         \right)S^{-1}(r)P^D.$$
First $n$ elements of this vector, corresponding to $\omega(t)-\overline\omega(t)$ are zero.
\end{proof}

\begin{corollary}
Vector function $\omega(t)$ has the following form:
$$\omega(t)=Y_N(r)e^{\overline\Lambda(r)\rho t}-\omega^*,$$
where
$$Y=S(r)\overline\Lambda(r)\diag (S^{-1}(r)P^D)$$
\end{corollary}
\begin{proof}
From (\ref{replace})
$$\overline x(t) = \int_0^te^{\overline A(r)(t-\tau)}P^Dd\tau=S(r)\overline\Lambda^{-1}(r)e^{\overline\Lambda(r)t}S^{-1}(r)P^D-\overline A^{-1}(r)p^D = $$
$$ = S(r)\overline\Lambda^{-1}(r)\diag (S^{-1}(r)p^D)e^{\overline\Lambda(r)\rho t}-\overline A^{-1}(r)P^D.$$
Since $(r)\lim_{t\rightarrow\infty}(S\overline\Lambda^{-1}(r)\diag (S^{-1}(r)P^D))_N e^{\overline\Lambda(r)\rho t}=0$ we have $(\overline A^{-1}(r)P^D)_N=\omega^*$. As is shown in the previous lemma, $\omega(t)$ does not depend on the bottom right block of $\Lambda(r)$, therefore $$(S(r)\overline\Lambda^{-1}(r)\diag (S^{-1}(r)P^D))_N=(S(r)\overline\Lambda(r)\diag (S^{-1}(r)P^D))_N.$$ Combining this facts we have
$$\omega(t)=\overline\omega(t)=$$$$=(S\overline\Lambda^{-1}(r)\diag (S^{-1}(r)P^D))_N e^{\overline\Lambda(r)\rho t}-(\overline A^{-1(r)}P^D)_N=(S(r)\overline\Lambda(r)\diag (S^{-1}(r)P^D))_N e^{\overline\Lambda(r)\rho t}-(\overline A^{-1}(r)P^D)_N=$$$$=Y_N(r) e^{\overline\Lambda(r)\rho t}-\omega^*.$$
\end{proof}

\begin{theorem}\label{t41}
For the frequency deviation $\omega(t)$ the following estimation exists:
$$|\omega(t)|\le|Y_N(r)|e^{\Real\,\Lambda(r) \rho t}+|\omega^*|$$
\end{theorem}
\begin{proof}
From
$$|\omega(t)-\omega^*|=Y_N(r) e^{ \Lambda(r)\rho t}+\omega^*-\omega^*=|Y_N(r) e^{ \Lambda(r) \rho t}|\le|Y_N(r)|e^{\Real  \Lambda(r) \rho t}.$$
We have
$$|\omega(t)-\omega^*|\le M^1(t,r)=|Y_N(r)|e^{\Real\Lambda(r) \rho t}$$
which gives theorem result.
\end{proof}

\begin{corollary}
$$|\omega(t)|\le|Y_N(r)|e^{\Real\,\Lambda(r) \rho t}+|\omega^*|.$$
\end{corollary}

\begin{theorem}\label{t42}
For the frequency deviation $\omega(t)$ in (\ref{meq}) we have the following estimation:
$$|\omega_l(t)|\le M_l^2(t,r)=\sum_{j=1}^{3n+m}|\Imag y_{lj}(r)|e^{\Real \lambda_j(r)t}\min\{|\Imag \lambda_j(r)t|,1\}+$$
$$+\sum_{j=1}^{3n+m}|\Real y_{lj}(r)|\min\left\{\left|\frac{d}{dt}f_j(t^1_j,r)\right|t,|f_j(t_j^0,r)-1|\right\},$$
where
\begin{equation*}
    f_j(t,r)=e^{\Real \lambda_j(r)t}\cos(\Imag \lambda_j(r)t)-1,
\end{equation*}
\begin{equation*}
    t_j^0=\left\{\begin{array}{cc}
    \frac{\pi}{|\Imag\lambda_j(r)|}+\frac{1}{\Imag\lambda_j(r)}\arctan{\frac{\Real \lambda_j(r)}{\Imag \lambda_j(r)}},&\mbox{if }\lambda_j(r)\ne0,\\
    0&\mbox{otherwise,}\end{array}\right.
\end{equation*} 
\begin{equation*}
    t^1_j=\left\{\begin{array}{cc}
    0,&\mbox{if }\Real \lambda_j(r)\ne0\mbox{ or }\Imag \lambda_j(r) = 0,\\
    \frac{\pi}{|\Imag\lambda_j(r)|}+\frac{1}{|\Imag\lambda_j(r)|}\arctan\left(\frac{(\Real \lambda_j(r))^2-(\Imag \lambda_j(r))^2}{2\Real \lambda_j(r)\Imag \lambda_j(r)}\right),
    &\mbox{if }(\Real \lambda_j(r)\ne0\-\Imag \lambda_j(r))\Imag \lambda_j(r)\ge0,\\
    \frac{1}{|\Imag\lambda_j(r)|}\arctan\left(\frac{(\Real \lambda_j(r))^2-(\Imag \lambda_j(r))^2}{2\Real \lambda_j(r)\Imag \lambda_j(r)}\right),
    &\mbox{if }(\Real \lambda_j(r)\ne0\-\Imag \lambda_j(r))\Imag \lambda_j(r)<0,
    \end{array}\right.
\end{equation*}
\end{theorem}
\begin{proof}
We can use the following notation
$$\omega(t)=\omega^{\sin}(t)+\omega^{\cos}(t),$$
where
$$\omega^{\sin}(t)=Im(Y_N(r))e^{\Real\Lambda(r)t}\sin{(\Imag\Lambda(r)t)}\rho,\;\omega^{\cos}(t)=Re(Y_N(r))e^{\Real\Lambda(r)t}\cos{(\Imag\Lambda(r)t)}\rho+\omega^*.$$
We will approximate each of this function separately.
$$|\omega^{\sin}_l(t)|=\left|\sum_{j=1}^{3n+m}\Imag y_{lj}(r)e^{\Real\lambda_j(r)t}\sin{(\Imag\lambda_j(r)t)}\right|\le\sum_{j=1}^{3n+m}|\Imag y_{lj}(r)|e^{\Real\lambda_j(r)t}|\sin{(\Imag\lambda_j(r)t)}|\le$$
$$\le\sum_{j=1}^{3n+m}|y_{lj}(r)|e^{\Real\lambda_j(r)t}\min\{|\Imag\lambda_j(r)t|,1\}.$$

To approximate $\omega_c(t)$ we will use the following expression:
$$|\omega^{\cos}_l(t)|=\left|\sum_{j=1}^{3n+m}\Real y_{lj}f_j(t,r)\right|,$$
Since $\omega(0)=0$, and $\omega^{\sin}(0)=0$, we have $\omega^{\cos}(0)=0$.

Let $t_j^0$ be a solution of the problem
$$\max_{\tau>0}|f_j(\tau,r)-1|,$$
then
\begin{equation*}
    t_j^0=\left\{\begin{array}{cc}
    \frac{\pi}{|\Imag\lambda_j(r)|}+\frac{1}{\Imag\lambda_j(r)}\arctan{\frac{\Real \lambda_j(r)}{\Imag \lambda_j(r)}},&\mbox{if }\lambda_j(r)\ne0,\\
    0&\mbox{otherwise,}\end{array}\right.
\end{equation*} 
Let $t^1_j$ be the solution of the problem
$$\max_{\tau\ge0}\left|f'_j(\tau,r)\right|,$$
then
\begin{equation*}
    t^1_j=\left\{\begin{array}{cc}
    0,&\mbox{if }\Real \lambda_j(r)\ne0\mbox{ or }\Imag \lambda_j(r) = 0,\\
    \frac{\pi}{|\Imag\lambda_j(r)|}+\frac{1}{|\Imag\lambda_j(r)|}\arctan\left(\frac{(\Real \lambda_j(r))^2-(\Imag \lambda_j(r))^2}{2\Real \lambda_j(r)\Imag \lambda_j(r)}\right),
    &\mbox{if }(\Real \lambda_j(r)\ne0\-\Imag \lambda_j(r))\Imag \lambda_j(r)\ge0,\\
    \frac{1}{|\Imag\lambda_j(r)|}\arctan\left(\frac{(\Real \lambda_j(r))^2-(\Imag \lambda_j(r))^2}{2\Real \lambda_j(r)\Imag \lambda_j(r)}\right),
    &\mbox{if }(\Real \lambda_j(r)\ne0\-\Imag \lambda_j(r))\Imag \lambda_j(r)<0,
    \end{array}\right.
\end{equation*}
Then we have the following estimation
$$|\omega_c(t)|=\left|\int_0^t\frac{d}{dt}\left(\sum_{j=1}^{3n+m}\Real y_{lj}(r)f_j(\eta,r)d\eta\right)\right|\le
\sum_{j=1}^{3n+m}|\Real y_{lj}(r)|\left|\int_0^tf'_j(\eta,r)d\eta\right|=$$
$$=\sum_{j=1}^{3n+m}|\Real y_{lj}(r)|\min\left\{\int_0^t\left|f'_j(\eta,r)\right|d\eta,\left|\int_0^tf'_j(\eta,r)d\eta\right|\right\}=$$
$$=\sum_{j=1}^{3n+m}|\Real y_{lj}(r)|\min\left\{\int_0^t\max_{\tau\ge0}\left|f'_j(\tau,r)\right|d\eta,\left|f_j(t^0_j,r)-1\right|\right\}\le$$
$$\le\sum_{j=1}^{3n+m}|\Real y_{lj}(r)|\min\left\{\max_{\tau\ge0}\left|f'_j(\tau,r)\right|t,\left|f_j(t^0_j,r)-1\right|\right\}\le$$
$$\le\sum_{j=1}^{3n+m}|\Real y_{lj}(r)|\min\left\{\left|f'_j(t^1_j,r)\right|t,|f_j(t_j^0,r)-1|\right\}.$$
\end{proof}

\begin{corollary}\label{MC}
$$|\omega(t)|\le M(t,r)=\min\{M^1(t,r),M^2(t,r)\},$$
where
$$M^2(t,r)=(M_1^2(t,r),\dots,M_n^2(t,r))^T,\;M(t,r)=(M_1(t,r),\dots,M_n(t,r))^T,$$
\end{corollary}

Similarly to nadir $F(r)$, estimate of the nadir is denoted as
$$G(r)=\max_{l\in N}\max_{t\in[0,t_1]}M_l(t,r).$$
\section{Minimization of the nadir}\label{mn}
\subsection{Calculations of values $F(r)$ and $G(r)$}
In practice each functions $M_l(t,r),\;l\in N$ have unique maximum. Taking this observation into consideration, we assume, that $G(r)$ can be found using golden section algorithm.

For nadir calculation d.c. approximation of frequency deviations is used. Functions $\omega_l(t)$ can be represented in the following way:
$$\omega_l(t)=h_l(t)-q_l(t),\;l\in N,$$
where
$$h_l(t) = \omega_l(t)+\frac{1}{2}k_l(r)t^2,\;q_l(t)=\frac{1}{2}k_l(r)t^2.$$
Here $k_l$ is obtained as follows.
$$\frac{d^2}{dt^2}\omega_l(t)=\sum_{j=1}^{3n+m}
e^{\Real\,\lambda_j(r)t}((\Real\,\lambda_j(r))^2\Real\,y_{lj}(r)-$$$$-2\Real\,\lambda_j(r)\Imag\,\lambda_j(r)\Imag\,y_{lj}(r)-(\Imag\,\lambda_j(r))^2\Real\,y_{lj}(r))\cos\Imag\,\lambda_j(r)t+$$
$$\left.((\Imag\,\lambda_j(r))^2\Imag\,y_{lj}(r)-2\Real\,\lambda_j(r)\Imag\,\lambda_j(r)\Real\,y_{lj}(r)-(\Real\,\lambda_j)^2\Imag\,y_{lj}(r))\sin\Real\,\lambda_j(r)t\right)\le$$
$$\le\sum_{j=1}^{3n+m}\left(((\Real\,\lambda_j(r))^2\Real\,y_{lj}(r)-2\Real\,\lambda_j(r)\Imag\,\lambda_j(r)\Imag\,y_{lj}(r)-(\Imag\,\lambda_j(r))^2\Real\,y_{lj}(r))^2+\right.$$
$$\left.+((\Imag\,\lambda_j(r))^2\Imag\,y_{lj}(r)-2\Real\,\lambda_j(r)\Imag\,\lambda_j(r)\Real\,y_{lj}(r)-(\Real\,\lambda_j(r))^2\Imag\,y_{lj}(r))\right)^\frac{1}{2}=k_l(r).$$
Global maximum of each function $|\omega_l(t)|,\;l\in N$ is obtained using branch and bound method with concave overestimators and d.c. approximation. Convergence of the method is given in \cite{HT}. Based on this optimization process values of $F(r)$ are found.

\subsection{Optimization of $F(r)$ and $G(r)$}

Although it is known \cite{T}, that eigenvalues and eigenvectors are continuous functions of matrix entries, they cannot be calculated analytically. Therefore Hooke and Jeeves algorithm \cite{BSS} is used optimize both function $F(r)$ and $G(r)$ subject to the following constraints:
\begin{enumerate}
  \item Parameters $r$ must be non-negative.
  \item Parameters $r$ have to be chosen so, that system of differential equations will remain stable.
  \begin{equation}\label{st_con}
  \max_{l\in N}\lambda_loverestimators=0.
  \end{equation}
  \item System's matrix does not have purely imaginary eigenvalues.
  \item Oscillation of the frequency oscillations must decrease at reasonable speed. There is no standardized constraints on the frequency oscillations decreases, therefore here the following constraint is used:
  \begin{equation}
  \max_{l\in N}\left|\frac{\Real\,\lambda_l(r)}{\Imag\,\lambda_l(r)}\right|\ge\xi,\;\xi>0.
  \end{equation}
\end{enumerate}
If those constraints are violated, we take
$$F(r)=G(r)=\infty.$$
Points found by Hooke and Jeeves algorithm are denoted by $r^F$ and $r^G$.
\subsection{Optimization algorithms}
Optimization algorithm for $F(r)$ has the following form:
\begin{alg} Minimization of $F(r)$.
\begin{enumerate}
\item Currently applied (default) in the optimized power system vector of control parameters $r^0$ is taken as a starting point.
\item Hooke and Jeeves method is applied to optimize $F(r)$. Each value of $F(r)$ is calculated via d.c. optimization.
\end{enumerate}
\end{alg}
Algorithm for optimization $G(r)$ has the same form with only difference in usage of golden section method for calculation of values $G(r)$ instead of d.c. optimization. It is given below:
\begin{alg} Minimization of $G(r)$.
\begin{enumerate}
\item Currently applied (default) in the optimized power system vector of control parameters $r^0$ is taken as a starting point.
\item Hooke and Jeeves method is applied to optimize $F(r)$. Each value of $F(r)$ is calculated via golden section method.
\end{enumerate}
\end{alg}
\section{Numerical results}\label{num_res}
The algorithms were coded in Matlab. Computations were made in PC with Intel Core i7 /2.4GHz / 16GB. Numerical results are presented in the table \ref{sample-table}. Here column System shows for which power system experiments were held, $\gamma$ is taken $0.01$. Two systems are considered \cite{PST}: 3 generators system and system of New England, that consists of 10 generators. For each system $100$ test with different vectors of disturbances. Averaged results are presented in the table. Column Optimization time represents time in seconds, required for the Hooke and Jeeves method, applied to functions $F$ and $G$. Next column contains number of function ($F$ or $G$) calculations required for the method. Last column contains values of function $F$ (maximal absolute value of frequency deviations) at the starting point, after optimization of $F$ and after optimization of $G$. Although in the last case function $G$ is optimized, aim of the algorithm is to minimize maximal frequency deviations, therefore values of $F$ are provided. Parameters in starting point $r$ are taken from \cite{PST}.

Time, required for optimization of $F$ is bigger than for $G$, due to the fact that, during every calculation of $F$ algorithm has to solve $n$ d.c. optimization problems, while during calculation of $F$ golden section computations are required.
As can be seen from the table, optimization of $G$ might give better set of parameters, than optimization of $F$. Moreover less calculations of the function are required. This effect can be seen in figures \ref{maj_func} and \ref{exact_func}. Here New England system is considered, all parameters $r_l$ are frozen with the exception of the first two. Red points on figures \ref{maj_func} and \ref{exact_func} represent $G(r^G)$ and $F(r^G)$ respectively, green point on figure \ref{exact_func} represents $G(r^G)$. In numerical experiments $F(r^G)$ is global optimum of $F$, so optimization of majorants gives better result, than optimization of maximal values of frequency deviations directly. This happens due to the fact, that function $G$ is smoother. Consider subregion, containing both points $G(r^G)$ and $F(r^F)$ on figures \ref{maj_func_part} and \ref{exact_func_part}. Function $F$ in this case have unique minimum $F(r^F)$, while $G$ has local minimums, and $G(r^G)$ is one of them.

Dynamics of the frequencies and majorant for New England System are given in figures \ref{new_englad_default}, \ref{new_england_maj}, \ref{new_england_exac} and \ref{new_england_maj_no_suppr}. Figure \ref{new_englad_default} represents behavior of the system with starting values of the parameters. Figure \ref{new_england_maj} represents behavior of the system with starting values of the parameters. Figure \ref{new_england_exac} represents behavior of the system with starting values of the parameters. Figure \ref{new_england_maj_no_suppr} system dynamics, after optimization of $F$ without suppression of oscillations (\ref{oc_con}). As can be seen, here maximal frequency deviations are smaller, than in \ref{new_england_maj}, however they do not decay, during the observation time.

\begin{table}
\caption{Resutls of numerical experiments for 3 generators system and New England bus system.}
{\begin{tabular}{|c|C{1.7cm}|C{1.7cm}|C{1.2cm}|C{1.2cm}|C{1.2cm}|C{1.2cm}|C{1.2cm}|}
   \hline
    \multirow{2}{*}{System} & \multicolumn{2}{c|}{Optimization Time (seconds)} & \multicolumn{2}{c|}{Function calculations} & \multicolumn{3}{c|}{Nadir values}\\\cline{2-8}
     &$F$ &$G$ &$F$ &$G$ & $F(r^0)$ &$F(r^F)$ &$F(r^G)$\\\hline
    3 generators & 3.63 & 0.15 & 225 & 142 & 1.18 & 0.55 & 0.55 \\\hline
   New England & 20.7 & 1.5 & 225 & 142 & 2.67 & 1.55 & 1.45\\
   \hline
 \end{tabular}
}
\label{sample-table}
\end{table}

\begin{figure}
\begin{center}
\resizebox*{13cm}{!}{\includegraphics{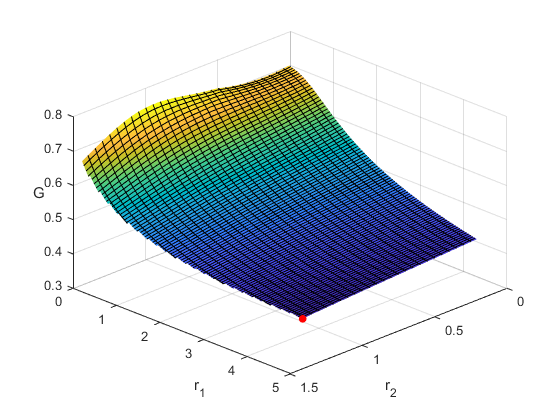}}
\caption{Function $G(r)$ for the New England system with $r_l,\;l=\overline{3,10}$ fixed. Red point represent value, found by the Hooke and Jeeves algorithm $G(r^G)$.} \label{maj_func}
\end{center}
\end{figure}

\begin{figure}
\begin{center}
\resizebox*{13cm}{!}{\includegraphics{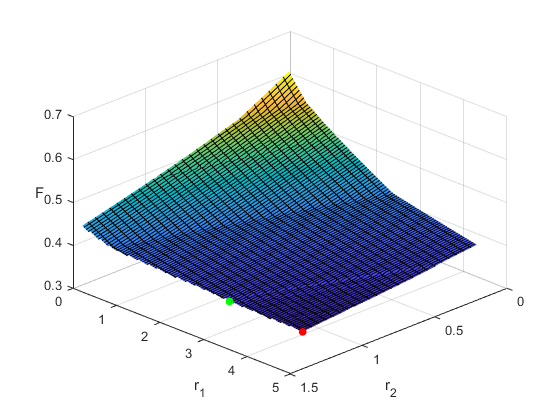}}
\caption{Function $F(r)$ for the New England system with $r_l,\;l=\overline{3,10}$ fixed. Green point represent value, found by the Hooke and Jeeves algorithm $F(r^F)$. Red point represents value $F(r^G)$.} \label{exact_func}
\end{center}
\end{figure}

\begin{figure}
\begin{center}
\resizebox*{13cm}{!}{\includegraphics{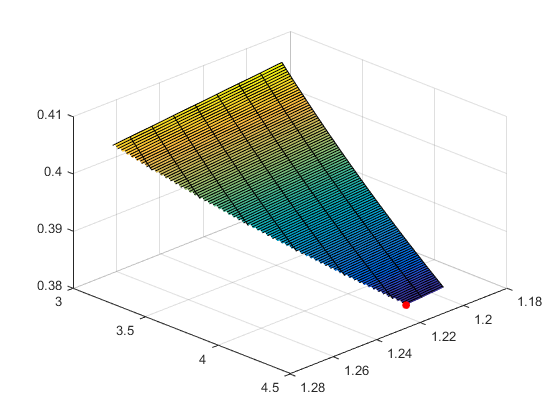}}
\caption{Function $G(r)$ for the New England system with $r_l,\;l=\overline{3,10}$ fixed. Subregion, containing optimal point. Red point represent value, found by the Hooke and Jeeves algorithm $G(r^G)$.} \label{maj_func_part}
\end{center}
\end{figure}
\begin{figure}
\begin{center}
\resizebox*{13cm}{!}{\includegraphics{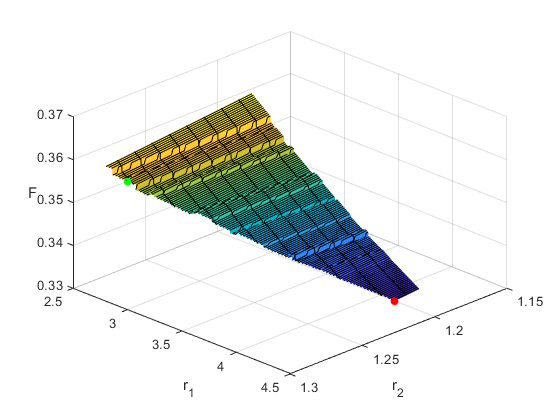}}
\caption{Function $F(r)$ for the New England system with $r_l,\;l=\overline{3,10}$ fixed. Subregion, containing optimal point. Green point represent value, found by the Hooke and Jeeves algorithm $F(r^F)$. Red point represents value $F(r^G)$.} \label{exact_func_part}
\end{center}
\end{figure}
\begin{figure}
\begin{center}
\resizebox*{13cm}{!}{\includegraphics{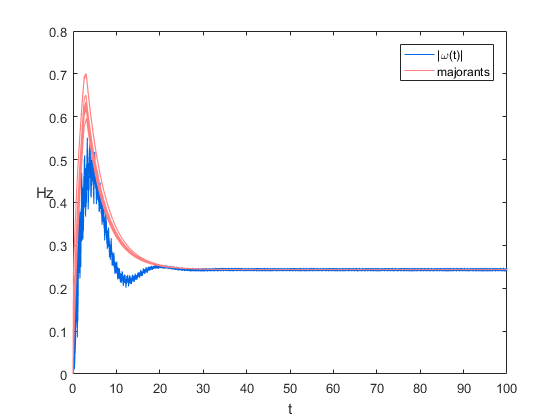}}
\caption{New England system. Absolute values of frequency deviations and corresponding majorants for starting values of $r$.} \label{new_englad_default}
\end{center}
\end{figure}
\begin{figure}
\begin{center}
\resizebox*{13cm}{!}{\includegraphics{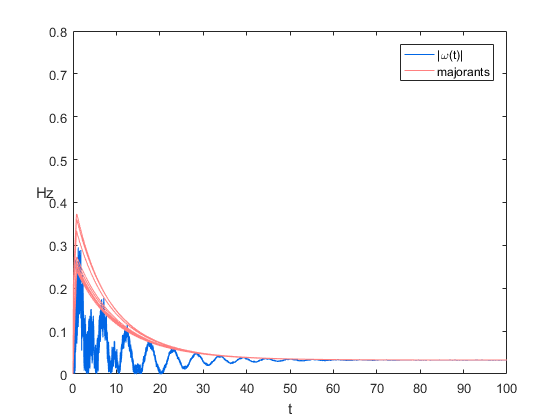}}
\caption{New England system. Absolute values of frequency deviations and corresponding majorants for $r^G$.} \label{new_england_maj}
\end{center}
\end{figure}

\begin{figure}
\begin{center}
\resizebox*{13cm}{!}{\includegraphics{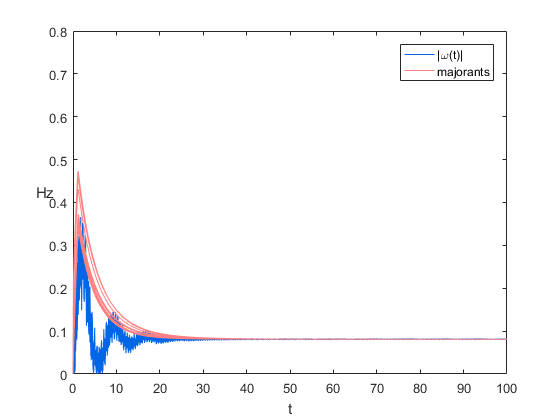}}
\caption{New England system. Absolute values of frequency deviations and corresponding majorants for $r^F$.} \label{new_england_exac}
\end{center}
\end{figure}
\begin{figure}
\begin{center}
\resizebox*{13cm}{!}{\includegraphics{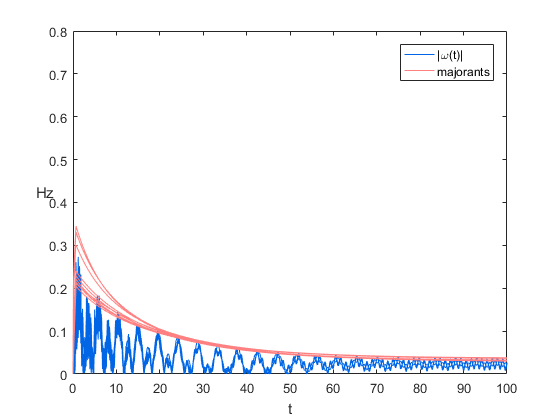}}
\caption{New England system. Absolute values of frequency deviations and corresponding majorants for $r^G$ without reduction of oscillations (\ref{oc_con}).} \label{new_england_maj_no_suppr}
\end{center}
\end{figure}

\section{Conclusion}\label{Conc}

The paper proposes method, that allows to minimize maximum of absolute values of frequency deviations (nadir) in power network under droop control. Control parameters are obtained by this method in a way, that ensures system's stability and good decay rate of the oscillations.

Since frequency oscillations are described by oscillatory functions with infinite number of local extremums, it is necessary to apply d.c. optimization method in order to find nadir for any particular set of control parameters, which is computationally expensive. Moreover nadir as a function of control parameters is non-smooth and have several local extremums. As a result it is not possible to ensure obtaining of global optimum.

Within this paper analytic majorants of frequency oscillations are derived. In practice they have unique maximum, therefore estimate of the nadir based on this majorants can be obtained via golden section, which is faster, than d.c. approach. Additionally, as numerical experiments show, this estimate as a function of control parameters have unique minimum which coincides with global minimum of the nadir.

Efficiency of nadir minimization and estimate minimization is compared. Unlike in nadir minimization, optimization algorithm applied to the estimate cannot finish its work in local minimum and always finds global optimum. Additionally minimization of estimate is less computationally difficult due to exclusion of d.c. optimization.

The obtained approach allows to reduce nadir up to two times compared with the default values of control parameters, used in the power networks. Moreover it keeps frequency oscillations within acceptable limits, which prevents wear of the equipment.

At the future expansion of this technique is planned in order to optimizer secondary frequency control together with the primary frequency control.

\end{document}